\pdfoutput=1
\documentclass[11pt]{amsart}
\usepackage{amssymb}
\usepackage{enumitem}
\usepackage[margin=1in]{geometry}
\usepackage[pdfdisplaydoctitle=true,
            colorlinks=true,
            urlcolor=blue,
            citecolor=blue,
            linkcolor=blue,
            pdfstartview=FitH,
            pdfpagemode= UseNone,
            bookmarksnumbered=true]{hyperref}
\usepackage{mathtools}
\usepackage{booktabs}
\usepackage{todonotes}
\usepackage{psfrag}
\usepackage{pst-pdf}
\usepackage{graphicx}

\newtheorem{theorem}{Theorem}
\newtheorem*{main*}{Main Theorem}
\newtheorem{conjecture}{Conjecture}

\newtheorem{corollary}{Corollary}
\newtheorem{lemma}{Lemma}

\theoremstyle{definition}

\newtheorem{remark}{Remark}

\newcommand{\norm}[1]{\left\Vert#1\right\Vert}

\DeclareMathOperator{\ad}{ad}

\DeclareMathOperator{\supp}{supp}
\DeclareMathOperator{\vol}{vol}

\newcommand{\Diff}{\mathrm{Diff}}

\newcommand{\e}{e}

\newcommand{\id}{\mathrm{id}}

\newcommand{\ex}{{\operatorname{ex}}}

\newcommand{\llangle}{\langle\!\langle}
\newcommand{\rrangle}{\rangle\!\rangle}

\mathtoolsset{showonlyrefs}

\setenumerate{label=(\alph*)}
\begin{document}

\title{Vanishing distance phenomena and the geometric approach to SQG}
\author{Martin Bauer}
\address{Department of Mathematics, Florida State University}
\email{bauer@math.fsu.edu}
\author{Philipp Harms}
\address{Department of Mathematics, Albert Ludwig University of Freiburg}
\email{philipp.harms@stochastik.uni-freiburg.de}
\author{Stephen C. Preston}
\address{Department of Mathematics, Brooklyn College of City University New York}
\email{stephen.preston@brooklyn.cuny.edu}
\thanks{MB was partially supported by a first year assistant professor award of the Florida State University. PH was supported by the Freiburg Institute of Advances Studies in the form of a Junior Fellowship. SCP was partially supported by Simons Foundation Collaboration Grant for Mathematicians no. 318969. SCP was also supported by a PSC-CUNY Award, jointly funded by The Professional Staff Congress and The City University of New York.}
\date{\today}

\begin{abstract}
In this article we study the induced geodesic distance of fractional order Sobolev metrics on the groups of (volume preserving) diffeomorphisms and symplectomorphisms. The interest in these geometries is fueled by the observation that they allow for a geometric interpretation for prominent partial differential equations in the field of fluid dynamics. These include in particular the modified Constantin--Lax--Majda and surface quasi-geostrophic equations. The main result of this article shows that both of these equations stem from a Riemannian metric with vanishing geodesic distance.
\end{abstract}

\maketitle

\tableofcontents

\section{Introduction}

It has been recently shown~\cite{Was2016} that the surface quasi-geostrophic (SQG) equation
\begin{equation*}
\theta_t+\langle u,\nabla \theta\rangle =0,\qquad
\theta=\nabla \times (-\Delta)^{-1/2}u
\end{equation*}
admits a geometric interpretation as the Euler--Arnold equation for geodesics of a right-invariant $H^{-1/2}$ metric on the group of diffeomorphisms which preserve the volume form of a two-manifold.
Recall that geodesics are critical points of the path length functional, and that the geodesic distance is the infimal length of paths between two given points.
In the article~\cite{Was2016} Washabaugh conjectured that the geodesic distance of the right-invariant $H^{-1/2}$ metric on the group of volume preserving diffeomorphisms is degenerate, i.e., there are distinct volume preserving diffeomorphisms whose geodesic distance is zero.
The main result of this article gives an affirmative answer to this conjecture:

\begin{main*}
Let $M$ be a two-dimensional orientable manifold with Riemannian metric $g$ and volume form $\mu=\vol(g)$, and let $\Diff_{\mu}(M)$ denote the group of all diffeomorphisms $\varphi$ satisfying $\varphi^*\mu=\mu$. Then the geodesic distance of the right-invariant $H^{-1/2}$ metric on $\Diff_\mu(M)$ is degenerate.
\end{main*}

This result is proven in Corollary~\ref{cor:volume} using the more general Theorem~\ref{thm:maintheorem}. We next discuss the relevance of this result in the broader context of the study of partial differential equations by geometric methods.

\subsection{A geometric view on partial differential equations}
Washabaugh's work stands in the tradition of studying partial differential equations (PDEs) from a geometric perspective by representing them as related to geodesic equations under suitable metrics. Generally on a Lie group $G$ with right-invariant metric, the geodesic equation 
splits into the decoupled pair of equations
\begin{equation*}
\partial_t g(t) = dR_{g(t)}u(t)\;, \qquad \partial_t u(t) + \ad_{u(t)}^{\star}u(t) = 0\;, \qquad g(t) \in G\;, u(t) \in T_eG\;;
\end{equation*}
here the first equation is the flow equation, while the second is called the \emph{Euler--Arnold equation}. 
This program for PDEs was started by Arnold \cite{Arn1966}, who represented Euler's equation of hydrodynamics as the Euler--Arnold equation of the right-invariant $L^2$ metric on the group of volume preserving diffeomorphisms.
Subsequently, similar representations were found for many other important PDEs in hydrodynamics and physics, including the modified Constantin--Lax--Majda equation \cite{CLM1985,Wun2010,EKW2012,BKP2016}, the Camassa--Holm equation \cite{CH1993,Mis1998,Kou1999}, the Korteweg--de Vries equation \cite{KO1987}, and the Hunter--Saxton equation \cite{HS1991,KM2003,Len2007,Len2008} (see \cite{Viz2008,Kol2017,KLMP2013} for surveys and further examples).
These representations allow one to study properties of the PDE in relation to properties of the underlying Riemannian manifold.
For example, local well-posedness of the PDE, including continuous dependence on initial conditions, is closely related to smoothness of the geodesic spray~\cite{EM1970}; see \cite{MP2010} for further results on smoothness for other Euler--Arnold equations.
Further geometric properties which have been studied in this context are the sign of the sectional curvature~\cite{KLMP2013}, Fredholmness of the exponential map~\cite{MP2010} and, as in this article, degeneracy of the geodesic distance functional \cite{MM2005}.

\subsection{Degeneracy of the geodesic distance on diffeomorphism groups}
\label{sec:intro:degeneracy}

One of the best-known instances of this phenomenon was discovered by Eliashberg and Polterovich \cite{ElPo1993}, who showed the degeneracy of the geodesic distance of the \emph{bi-invariant} $W^{-1,p}$ metric with $p<\infty$ on the group of symplectomorphisms.
This is in contrast to Hofer's $W^{-1,\infty}$ metric, which has non-degenerate geodesic distance \cite{hofer1990topological}.
More than ten years later, Michor and Mumford \cite{MM2005} proved that the geodesic distance of the \emph{right-invariant} $L^{2}$ metric on the group of diffeomorphisms vanishes identically. Here the corresponding Euler--Arnold equation is the inviscid Burgers' equation $u_t+3uu_x=0$. Subsequently, Bauer, Bruveris, Harms and Michor \cite{BBHM2012,BBHM2013,BBM2013} extended this result to fractional order Sobolev metrics of order $s< 1/2$ on general diffeomorphism groups and for $s=\frac12$ on the diffeomorphism group of the circle, and to the $L^2$ metric on the Virasoro--Bott group, whose Euler--Arnold equation is the Korteweg--de Vries equation. While the reasons for this degeneracy are still mysterious, it has been conjectured by Michor and Mumford \cite{MM2005} that there exists a relation  to locally unbounded curvature of the corresponding Riemannian metric.

In many cases the (non-)degeneracy of the geodesic distance goes hand in hand with Fredholmness of the exponential map and well-posedness properties of the geodesic equation and Euler--Arnold equation. Sobolev metrics on diffeomorphism groups depend on a smoothness parameter $s$, the number of derivatives of the vector field that appear in the metric at the identity, and the higher this parameter is, the better-behaved geodesics are. For right-invariant Sobolev metrics of fractional order on the diffeomorphism group of a one-dimensional manifold, we summarize the known geometric properties in Table~\ref{tab:1} below: smoothness of the exponential map $u_0\mapsto g(1)$, Fredholmness of this map, global existence of geodesics, and nonvanishing geodesic distance.

\begin{table}[h]
\centering
\begin{tabular}{ccccccccc}
\toprule
$s$ & $0\leq s<\tfrac{1}{2}$ & $s=\tfrac{1}{2}$ & $\tfrac{1}{2}<s<\tfrac{3}{2}$ & $s> \tfrac{3}{2}$ \\ \midrule
smoothness & False \cite{CK2002} & True \cite{EKW2012,CK2002} & True \cite{EK2014a} & True  \cite{EK2014a} \\ 
Fredholmness & False & False \cite{BKP2016} & True\footnotemark{} \cite{MP2010} & True\addtocounter{footnote}{-1}\footnotemark{} \cite{MP2010} \\ 
global exist. & False for $s=0$ & False \cite{CC2010,BKP2016,PW2018} & False for $s=1$ \cite{CH1993, McK1998, CE1998b} & True \cite{PW2018,EK2014a} \\ 
nonvanishing & False \cite{BBHM2013} & {\bf False} \cite{BBHM2013} & True \cite{BBHM2013} & True \cite{BBHM2013}
\\\bottomrule
\end{tabular}
\caption{Geometric properties of $H^s$ metrics on $\Diff(S^1)$ and $\Diff(\mathbb R)$. For the case $\Diff(\mathbb R)$ the bold statement is a new contribution of this article.}
\label{tab:1}
\end{table}
\footnotetext{\label{foot:1}The arguments in \cite{MP2010} for Fredholmness can be extended to fractional orders, assuming that smoothness of the metric and spray holds true.}

Clearly the case $s=\tfrac{1}{2}$ is the transition for most of these properties, which suggests that there are some connections between them. 
Global existence is known only for orders $s\in\{0,\frac12,1\}$, where it fails, and for orders $s>\tfrac{3}{2}$, where global existence holds almost trivially because the Riemannian distance generates the manifold topology, and standard results of Riemannian geometry on Hilbert manifolds apply \cite{Lan1999}. 

For diffeomorphism groups on higher dimensional manifolds the critical indices for Fredholmness and smoothness of the exponential map do not change, whereas the critical indices for vanishing geodesic distance and global existence depend on the dimension. Vanishing geodesic distance for $\frac12\leq s<1$ is an extremely recent result by Jerrard and Maor \cite{JM2018}, who disproved an earlier conjecture by Bauer, Bruveris, Harms, and Michor \cite{BBHM2013}. We again summarize the known geometric properties in Table~\ref{tab:1a} below:
\begin{table}[h]
\centering
\begin{tabular}{ccccccccc}
\toprule
$s$ & $0\leq s<\tfrac{1}{2}$ & $s=\tfrac{1}{2}$ & $\tfrac{1}{2}<s<1$ & $s\geq 1$ \\ \midrule
smoothness & False & True\footnotemark{} \cite{BEK2014} & True \cite{BEK2014} & True  \cite{BEK2014} \\ 
Fredholmness & False & False & True \cite{MP2010} & True \cite{MP2010} \\ 
global exist. & False for $s=0$ & False  & False for $s=1$ & True for $s>\frac{d}{2}+1$ \cite{BEK2014,BV2017,MP2010} \\ 
nonvanishing & False \cite{BBHM2013} & False \cite{JM2018} & False \cite{JM2018} & True \cite{MM2005}
\\\bottomrule
\end{tabular}
\caption{Geometric properties of $H^s$ metrics on $\Diff(M)$ for a manifold $M$ of dimension $d\geq 2$.}
\label{tab:1a}
\end{table}
\footnotetext{For fractional order metrics  smoothness and global existence results have only been shown for the case $M=\mathbb R^d$.}

For the volume-preserving diffeomorphism group of a simply-connected compact two-dimensional surface $M$, the critical exponents change again, as can be seen from Table~\ref{tab:2} below. The geodesic distance was previously known to be nondegenerate for $s=0$ (corresponding to ordinary 2D Euler) and degenerate for $s=-1$. This paper completes the picture for $s\le -\tfrac{1}{2}$. For the interval $-\tfrac{1}{2}<s<0$, the answer is still unknown. 

\begin{table}[h]
\begin{tabular}{cccccccccc}
\toprule
$s$ &$s=-1$& $-1<s<-\tfrac{1}{2}$ & $s=-\tfrac{1}{2}$ & $-\tfrac{1}{2}<s$ \\ \midrule
smoothness & False & False & True \cite{Was2016} & True \cite{MP2010} \\ 
Fredholmness & False & False & False \cite{Was2016} & True \cite{EMP2006, MP2010} \\ 
global exist. & True\footnotemark & unknown & unknown & True for $s=0$ \cite{Wol1933} and $s=1$ \cite{Shk2001} \\ 
nonvanishing & False \cite{ElPo1993} & \textbf{False} & \textbf{False} & True for $s\ge 0$ \cite{MM2005}
\\\bottomrule
\end{tabular}
\caption{Geometric properties of $H^s$ metrics on $\Diff_\mu(M)$, where $M$ is a closed surface. The bold statements are new contributions of this article.}
\label{tab:2}
\end{table}

\footnotetext{Global existence for $s=-1$ holds for an entirely different reason (bi-invariance of the metric, which implies that the Riemannian exponential coincides with the group exponential) than for metrics of order $s\geq 0$ (PDE methods, which imply in addition the smoothness of the exponential map).}

In higher dimensions the critical exponents changes again. In dimension 3 it becomes $s=0$ for both Fredholmess and vanishing, corresponding to the usual 3D Euler equation. Here the exponential map is smooth~\cite{EM1970} but not Fredholm~\cite{EMP2006}, while global existence is notoriously unknown. The geodesic distance is positive, but not due to an intrinsic property of the metric: rather due to the fact that the right-invariant metric happens to be the restriction of the non-invariant metric, for which the geodesics are known explicitly and given by pointwise geodesics in the base manifold $M$~\cite{EM1970}. The completion of the smooth volume-preserving diffeomorphisms in the Riemannian distance is the space of all measure-preserving maps, a result in dimension 3 and higher due to Shnirelman~\cite{Shn1987}. Intuitively we may think of the volume-preserving constraint as doing very little to enforce smoothness in dimension $3$ or higher; on the other hand in two dimensions the completion is smaller (though it is not known exactly what it is). The fact that there are smooth volume-preserving diffeomorphisms in a 3D cube which cannot be joined by a minimizing geodesic, and that the diameter of this group is finite~\cite{Shn1994} is further evidence that for 3D fluids, the distance is ``nearly'' degenerate. From the tables above, we may suspect that these geometric properties are related to each other and to the global existence question, though as yet no direct implication is known. 

\subsection{Relation to degeneracy of the displacement energy}

This article simplifies and unifies the methods which were used by Michor, Mumford, Bauer, and Bruveris \cite{MM2005,BBHM2012,BBHM2013,BBM2013} to prove degeneracy of the geodesic distance on diffeomorphism groups.
One key insight is the observation that an argument of Eliashberg and Polterovich~\cite{ElPo1993}, which links degeneracy of the geodesic distance to degeneracy of the displacement energy, generalizes from bi-invariant to right-invariant metrics; see Theorem~\ref{thm:displacement}.
This significantly widens the applicability of \cite{ElPo1993}, as it allows us to study the large class of right-invariant Sobolev metrics on diffeomorphism groups. 
In the context of $W^{-1,p}$-norms on the contactomorphism group this has been observed by Shelukhin \cite[Remark~7]{shelukhin2017hofer}. 
We present a formal proof of this result in the context of general groups of transformations.

These results circumvent the main difficulty in the proofs of vanishing geodesic distance of \cite{MM2005,BBHM2012,BBHM2013,BBM2013,JM2018}, namely, to construct short paths of diffeomorphisms with fixed end points.
In contrast, there is no end point constraint in the definition of the displacement energy, other than that some fixed set of points has to be mapped to some disjoint location.
This is much easier to handle.

\subsection{Application to Sobolev metrics on diffeomorphism groups}

We show that the geodesic distance of the $H^{1/2}$ metric on diffeomorphism groups vanishes identically; see Theorem~\ref{thm:vanishingDiff}.
The corresponding Euler--Arnold equation is the Wunsch (modified Constantin--Lax--Majda) equation \cite{CLM1985,Wun2010,EKW2012,BKP2016}.
Moreover, we show that the geodesic distance of the $H^{-1/2}$ metric on groups of exact diffeomorphisms vanishes identically; see Theorem~\ref{thm:maintheorem}.
This implies the degeneracy of the geodesic distance on groups of volume preserving diffeomorphisms on two-manifolds; see Corollary~\ref{cor:volume}.
The corresponding Euler--Arnold equation is the SQG equation.
We conjecture that these results are sharp, referring to Section \ref{sec:conjecture2} for precise statements.

\subsection{Structure of the article}

Section~\ref{sec:displacement} contains the characterization of the degeneracy of the geodesic distance in terms of the displacement energy. Sections~\ref{sec:diffeomorphisms} and~\ref{sec:volumepreserving} contain applications of this theorem to groups of diffeomorphisms and volume preserving diffeomorphisms, respectively.

\section{Right-invariant Riemannian metrics on Lie groups}\label{sec:displacement}
In this section we establish a necessary and sufficient condition for the (non-)degeneracy of the geodesic distance on infinite-dimensional  groups with right-invariant weak Riemannian metrics.
This setting is natural for the study of diffeomorphism groups and other infinite-dimensional topological groups; see the applications in Sections~\ref{sec:diffeomorphisms}--\ref{sec:volumepreserving}.

\subsection{Geodesic distance}\label{sec:geo}

Let $G$ be a (possibly infinite dimensional) manifold and topological group with neutral element $\e$, Lie algebra $\mathfrak g=T_eG$, and left and right translations $L$ and $R$ given by
\begin{equation}
g_1 g_2= L_{g_1}(g_2)=R_{g_2}(g_1),\;\qquad \forall g_1,g_2\in G\;.	
\end{equation}
Assume for each $g\in G$ that $R_g\colon G\to G$ is smooth, and let $\llangle \cdot,\cdot\rrangle$ be an inner product on the Lie algebra $\mathfrak g$.
This gives rise to the following right-invariant Riemannian metric on $G$:
\begin{equation}
\llangle h_1,h_2 \rrangle_g = \llangle TR_{g^{-1}}h_1,TR_{g^{-1}}h_2\rrangle,\qquad \forall g\in G,\; \forall h_1,h_2\in T_g G\;.
\end{equation}
The corresponding geodesic distance function is defined as
\begin{align}
d(g_1,g_1)={\operatorname{inf}}   \int_0^1 \llangle \partial_t g(t),\partial_t g(t) \rrangle_{g(t)} dt\;,\;\qquad \forall g_1,g_2\in G\;,
\end{align}
where the infimum is taken over all smooth paths in $G$ with $g(0)=g_0$ and $g(1)=g_1$.
The geodesic distance function is called degenerate if $d(g_1,g_2)=0$ for some $g_1\neq g_2 \in G$, and it is called vanishing if $d(g_1,g_2)=0$ for all $g_1, g_2 \in G$.

\subsection{Displacement energy}\label{sec:dis}

Assume the setting of Section~\ref{sec:geo}, and let $G$ act effectively and continuously from the left on a set $M$.
Then the displacement energy \cite{ElPo1993} of a subset $A\subseteq M$ is defined as
\begin{align}
E(A)=\inf\left\{d(\e,g):g\in G, g(A)\cap A=\emptyset \right\}\;,
\end{align}
the support of a transformation $g \in G$ is defined as
\begin{align}
\supp(g)=\left\{x\in M: g(x)\neq x \right\},
\end{align}
and the group of transformations with support in $A \subseteq M$ is denoted by \cite{ElPo1993}
\begin{align}
G_A=\left\{g\in G: \operatorname{supp}(g)\subset A \right\}.
\end{align}
A subset $A\subseteq M$ is called essential if the corresponding group $G_A$ is non-Abelian,
and a transformation $g \in G$ is called non-trivial if $g\neq e$.

\subsection{Relation between geodesic distance and displacement energy}\label{sec:rel}
On finite-dimensional manifolds and, more generally, manifolds with strong Riemannian metrics, the geodesic distance is always non-degenerate \cite{Lan1993}.
For weak Riemannian metrics this is no longer true: there exists Riemannian metrics that induce vanishing geodesic distance \cite{ElPo1993,MM2005}.
In this section we will describe an equivalence between this degeneracy of the geodesic distance and degeneracy of the displacement energy. This result is a generalization of a result  by Eliashberg and Polterovich \cite{ElPo1993} for 
the group of symplectomorphisms with bi-invariant weak Riemannian metric.

The scarcity of bi-invariant metrics in the context of infinite dimensional Lie-groups limits the applicability of their result. Theorem~\ref{thm:displacement} shows that left-invariance is not needed and can be replaced by condition~\eqref{ass:lipschitz}, which holds automatically for all bi-invariant metrics (in this case, the constant $|L_g|$ is equal to one). In the context of 
the contactomorphism group this result has been already observed by Shelukhin in  \cite{shelukhin2017hofer}. In the following we will formulate the result for a general group of transformations acting on a set $M$. The proof follows the sketch of Shelukhin, which is based on an adaption of the original argument by Eliashberg and Polterovich, see \cite{shelukhin2017hofer,ElPo1993}. 

\begin{theorem}\label{thm:displacement}
Assume the setting of Sections~\ref{sec:geo}-\ref{sec:dis}, and assume for each $g \in G$ that left translation by $g$ is Lipschitz continuous:
\begin{equation}\label{ass:lipschitz}
|L_g| := \inf\left\{C \in \mathbb R_+: d(gg_0,gg_1) \leq C d(g_0,g_1),  \forall g_0,g_1 \in G \right\} <\infty\;.
\end{equation}
Furthermore assume that every non-empty, open subset $A\subset M$ is essential.

Then the following three statements are equivalent:
\begin{enumerate}
\item\label{item:displacement1} There exists a non-trivial transformation $G\ni g \neq e$ with $d(\e,g)=0$.
\item\label{item:displacement2} There exists a normal subgroup of transformations $g \in G$ with $d(\e,g)=0$ which contains at least one non-trivial transformation $g\neq e$.
\item\label{item:displacement3} There exists an open set $A\subseteq M$ with displacement energy $E(A)=0$.
\end{enumerate}
If $G$ is a simple group then any of the above statements imply
\begin{enumerate}[resume]
\item\label{item:displacement4} The geodesic distance function vanishes identically, i.e., $d(g_1,g_2)=0$ for all $g_1,g_2\in G$.
\end{enumerate}
\end{theorem}

\begin{proof}
For brevity, we write $\norm{h}_g=\sqrt{\llangle h,h\rrangle_g}$ for all $g \in G$ and $h \in T_gG$.
One easily verifies that the geodesic distance is symmetric and satisfies the triangle inequality, i.e.,
\begin{equation*}
d(g_2,g_1)=d(g_1,g_2)\leq d(g_1,g)+d(g,g_2),
\qquad
\forall g,g_1,g_2 \in G.
\end{equation*}
Moreover, the invariance properties of the metric imply that
\begin{align*}
d(g_1g,g_2g)=d(g_1,g_2),
\qquad
d(gg_1,gg_2)\leq |L_g| d(g_1,g_2),
\qquad
\forall g,g_1,g_2 \in G.
\end{align*}

\noindent\ref{item:displacement1} $\implies$ \ref{item:displacement2}:
Let $G^0$ be the set of all transformations $g\in G$ with $d(\e,g)=0$.
Then $G^0$ is a subgroup of $G$ because it holds for each $g_1, g_2 \in G^0$ that
\begin{align*}
d(\e,g_1g_2^{-1})\leq d(\e,g_2^{-1})+d(g_2^{-1},g_1g_2^{-1})=d(g_2,\e)+d(\e,g_1)=0.
\end{align*}
Moreover, $G^0$ is a normal subgroup of $G$ because it holds for all $g_0 \in G^0$ and $g \in G$ that
\begin{equation*}
d(\e,gg_0g^{-1})=d(g,gg_0)\leq |L_g| d(\e,g_0)=0.
\end{equation*}
Thus, $G_0$ is a normal subgroup of $G$, which contains a non-trivial transformation by \ref{item:displacement1}, and we have shown \ref{item:displacement2}.

\noindent
\ref{item:displacement2} $\implies$ \ref{item:displacement3}: Let $g$ be an non-trivial transformation in $G^0$. As $g$ is non-trivial, there exists an open set $A\subseteq M$ such that $g(A)\cap A=\emptyset$ (recall that we assumed the action of $G$ to be continuous). Together with $g\in G^0$ this implies $E(A)=0$, which proves \ref{item:displacement3}.

\noindent
\ref{item:displacement3} $\implies$ \ref{item:displacement1}: This generalizes the proof for bi-invariant metrics in \cite{ElPo1993} and is similar to the proof described in \cite{shelukhin2017hofer}. The main ingredient is the following estimate for the distance of the commutator $[g_0,g_1]:=g_0^{-1}g_1^{-1}g_0g_1$ of $g_0,g_1 \in G$ to the neutral element:
\begin{equation}\label{ineq:commutator}\begin{aligned}
d(\e,[g_0,g_1])
&\leq
\operatorname{min}\left((1+|L_{g_0^{-1}}|) d(\e,g_1),
(1+|L_{g_1^{-1}}|) d(\e,g_0)\right)\;.
\end{aligned}\end{equation}
Note that \eqref{ineq:commutator} is trivially satisfied for bi-invariant metrics \cite{ElPo1993}. To prove \eqref{ineq:commutator} we calculate
\begin{equation}\begin{aligned}
d(\e,[g_0,g_1])
&=
d(\e,g_0^{-1}g_1^{-1}g_0g_1)
=
d(g_1^{-1}g_0^{-1} ,g_0^{-1}g_1^{-1})
\leq
d(g_1^{-1}g_0^{-1},g_0^{-1})+ d(g_0^{-1},g_0^{-1}g_1^{-1} )
\\&\leq
d(g_1^{-1},\e)+ |L_{g_0^{-1}}| d(\e,g_1^{-1} )
=
(1+|L_{g_0^{-1}}|) d(\e,g_1),
\end{aligned}\end{equation}
and
\begin{equation}\begin{aligned}
d(\e,[g_0,g_1])
&=
d(\e,g_0^{-1}g_1^{-1}g_0g_1)
=
d(g_1^{-1}g_0^{-1} ,g_0^{-1}g_1^{-1})
\leq
d(g_1^{-1}g_0^{-1},g_1^{-1})+ d(g_1^{-1},g_0^{-1}g_1^{-1} )
\\&\leq
|L_{g_1^{-1}}| d(\e,g_0^{-1} )
+d(g_0^{-1},\e)
=
(1+|L_{g_1^{-1}}|) d(\e,g_0).
\end{aligned}\end{equation}
Inequality \eqref{ineq:commutator} allows us to reuse the proof of \cite{ElPo1993} to show the degeneracy of the metric.
Therefore let $A$ be a non-empty, open set with zero displacement energy.
As $G_A$ is non-Abelian, we can choose $g_0,g_1 \in G_A$ with $[g_0, g_1]\neq\e$.
For any $g_2 \in G$ with $g_2(A)\cap A=\emptyset$ we let $g_3= g_0g_2^{-1}g_0^{-1}g_2=[g_0^{-1},g_2]$.
Then it holds for all $x \in M$ that
\begin{equation}\label{equ:g0g1g3}
g_0^{-1}g_1g_0x=g_3^{-1}g_1g_3x\;.
\end{equation}
For $x\in A$ this is obvious because $g_3=g_0$ on $A$.
For $x\notin A$ and $g_2(x)\notin A$ we have $g_3(x)= x=g_0(x)$ and thus it is true as well.
It remains to check the case $x\notin A$ and $g_2(x)\in A$. Then $g_3(x)= g_2^{-1}g_0^{-1}g_2(x)\notin A$.
Here we used  that $g_2^{-1}(A)\cap A=\emptyset$ and that $g_0(x)=x$ on $M\setminus A$.
Thus, $g_1g_3(x)=  g_2^{-1}g_0^{-1}g_2(x)=g_3(x)$ and $g_0^{-1}g_1g_0(x)=g_3^{-1}g_1g_3(x)=x$, which proves \eqref{equ:g0g1g3} for all $x\in M$. As $G$ acts effectively on $M$, it follows that $g_0^{-1}g_1g_0=g_3^{-1}g_1g_3$.
Therefore,
\begin{align*}
d(\e,[g_1, g_0])
&=
d(\e,g_1^{-1}g_0^{-1}g_1g_0)
=
d(\e,g_1^{-1}g_3^{-1}g_1g_3)
=
d(\e,[g_1,g_3])\leq
(1+|L_{g_1^{-1}}|)d(\e,g_3)
\\&=(1+|L_{g_1^{-1}}|)d(\e,[g_0^{-1},g_2])
\leq (1+|L_{g_1^{-1}}|)
(1+|L_{g_0}|)d(\e,g_2).
\end{align*}
Taking the infimum over all $g_2$ with $g_2(A)\cap A=\emptyset$ yields
\begin{align*}
d(\e,[g_1,g_0])\leq (1+|L_{g_1^{-1}}|)(1+|L_{g_0}|)  E(A)=0.
\end{align*}
Thus, we have shown that $[g_1, g_0]$ is a non-trivial transformation with $d(\e,[g_1,g_0])=0$, which proves \ref{item:displacement1}.
This completes the proof of the equivalence of \ref{item:displacement1}, \ref{item:displacement2}, and \ref{item:displacement3}.

\ref{item:displacement2} $\implies$ \ref{item:displacement4}:
Note that $G^{0}=G$ because the only non-trivial normal subgroup of a simple group is the group itself. Now the statement follows by the triangle inequality
\begin{equation*}
d(g_0,g_1)\leq d(g_0,\e)+d(\e,g_1)=0.
\qedhere
\end{equation*}
\end{proof}

\section{Diffeomorphism groups and the modified Constantin--Lax--Majda equation}\label{sec:diffeomorphisms}

\subsection{Sobolev metrics on diffeomorphism groups}\label{sec:sob}

Let $(M,\langle\cdot,\cdot\rangle)$ be a connected Riemannian manifold of bounded geometry.\footnote{That is, the injectivity radius of $(M,\langle\cdot,\cdot\rangle)$ is positive and each iterated covariant derivative of the curvature is uniformly bounded in the metric; see \cite{Greene1978,MueNar2015} for more details. This is automatically the case if $M$ is compact or Euclidean.} 
For fixed $s \in \mathbb R$, let $\llangle \cdot,\cdot\rrangle$ be a Sobolev inner product of order $s$ on the vector space $\mathfrak X(M)$ of compactly supported vector fields, and let $\|\cdot\|$ be the corresponding norm.
We omit the exact description here and refer the interested reader to the article \cite{BBHM2013} or the more extensive references \cite{Tri1983,Tri1992,Eichhorn2007}.
For our purposes it suffices to say that a Sobolev $H^s$-norm on real-valued functions $f$ on $\mathbb R^n$ is given by
\begin{equation}
\label{Hs_norm}
\| f \|_{H^s(\mathbb R^n)}^2 = \| \mathcal F^{-1} a \mathcal F f \|_{L^2(\mathbb R^n)}^2\;,
\end{equation}
where $\mathcal F$ is the Fourier transform
and $a\in C^{\infty}(\mathbb R^n)$ is a Fourier multiplier, which satisfies for some constants $C_1,C_2>0$ that
\begin{align}
C_1(1+|\xi|^2)^{\frac{s}2}\leq a(\xi)\leq C_2 (1+|\xi|^2)^{\frac{s}2}\;,\qquad\forall \xi \in \mathbb R^n\;.
\end{align}
This definition extends to vector fields on general manifolds via charts and partitions of unity.

Let $\Diff(M)$ be the connected component of the identity in the group of all smooth compactly supported diffeomorphisms of $M$.
Then $\Diff(M)$ is a convenient Lie group with Lie algebra $\mathfrak X(M)$ \cite{KM1997}.
The right-invariant $H^s$ metric on $\Diff(M)$ is defined as
\begin{equation*}
\llangle h,k\rrangle_{\varphi} = \llangle h\circ \varphi^{-1}, k\circ \varphi^{-1} \rrangle\;,\qquad\forall\varphi \in \Diff(M), \forall h, k \in T_\varphi \Diff(M)\;,
\end{equation*}
and the corresponding geodesic distance is defined as in Section~\ref{sec:geo}.

\subsection{Bump functions with small $H^{\frac12}$ norm}

An essential ingredient in the proof of the degeneracy of the geodesic distance of the $H^{\frac12}$ metric on $\Diff(M)$ (see Theorem~\ref{thm:vanishingDiff}) is the existence of bump functions with small $H^{\frac12}$ norm. In the following we will prove a slight refinement  of \cite[Lemma 3.3]{BBHM2013}, which is used several times in the remainder of the article. The construction is illustrated in Fig.~\ref{fig:bump}, and further details can be found in \cite[Theorem 13.2]{Tri2001}.

\begin{lemma}\label{lem:bumpfunction}
There exists a sequence $(\xi_n)_{n\in\mathbb N}$ in $ C^{\infty}(\mathbb R,[0,1])$ such that
\begin{enumerate}
\item $\xi_n(x)=1$ for all $x \in [-2^{-n},2^{-n}]$ and $n\in\mathbb N$,\label{prop1}
\item $\xi_n(x)=0$ for all $x\notin [-1,1]$ and $n\in\mathbb N$, and \label{prop2}
\item $\sup_{n\in\mathbb N} n \|\xi_n\|^2_{H^{1/2}(\mathbb R)}<\infty$.\label{prop3}
\end{enumerate}
\end{lemma}

\begin{figure}[h]
\centering
\begin{psfrags}
\psfrag{S0}[tc][tc]{$0$}%
\psfrag{S1}[tc][tc]{$1$}%
\psfrag{Sm1}[tc][tc]{$-1$}%
\psfrag{W0}[cr][cr]{$0$}%
\psfrag{W1}[cr][cr]{$1$}%
\psfrag{x0}[tc][tc]{$0$}%
\psfrag{x2}[tc][tc]{$0.2$}%
\psfrag{x4}[tc][tc]{$0.4$}%
\psfrag{xm2}[tc][tc]{$-0.2$}%
\psfrag{xm4}[tc][tc]{$-0.4$}%
\psfrag{y0}[cr][cr]{$0$}%
\psfrag{y1}[cr][cr]{$0.1$}%
\psfrag{y2}[cr][cr]{$0.2$}%
\psfrag{y3}[cr][cr]{$0.3$}%
\psfrag{y4}[cr][cr]{$0.4$}%
\psfrag{y5}[cr][cr]{$0.5$}%
\includegraphics[width=0.5\textwidth]{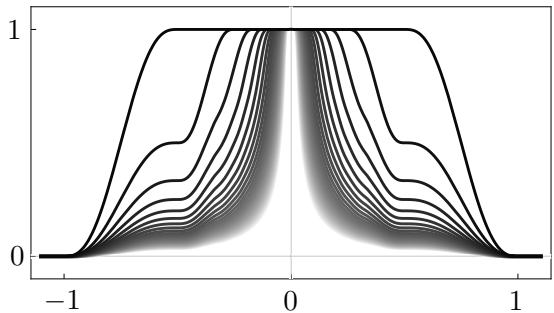}
\end{psfrags}
\caption{A sequence $(\xi_n)_{n\in\mathbb N}$ of bump functions with small $H^{1/2}$ norm constructed as in the proof of Lemma~\ref{lem:bumpfunction}.}
\label{fig:bump}
\end{figure}

\begin{proof}
Let $f\colon\mathbb R\to[0,1]$ be a smooth function with support in $[-1,1]$ such that $f(x)=1$ for all $x \in [-\frac12,\frac12]$.
For each $n \in \mathbb N$ let $\xi_n \colon \mathbb R \to [0,1]$ be given by
\begin{equation*}
\xi_n(x)=\frac{1}{n} \sum_{j=0}^{n-1}  f(2^j x)\;,\qquad\forall x \in \mathbb R.
\end{equation*}
Then $\xi_n$ obviously satisfies~\ref{prop1} and~\ref{prop2}.
By \cite[Lemma 3.3]{BBHM2013} it follows that $(\xi_n)_{n\in\mathbb N}$ satisfies~\ref{prop3}.
\end{proof}

\subsection{Vanishing geodesic distance on diffeomorphism groups}

Previous work by some of the authors \cite{BBHM2013,BBM2013} shows that the geodesic distance vanishes for $s<\frac{1}{2}$ on $\Diff(M)$ and for $s\leq\frac{1}{2}$ on $\Diff(S^1)$.
In these articles it was conjectured that the result extends to $s=\frac{1}{2}$ and general manifolds $M$. In the recent article \cite{JM2018} the vanishing geodesic distance result has been extended to metrics of order $\frac{1}{2}<s<1$ for $\operatorname{dim}(M)>1$. Thus only the case $M=\mathbb R$ and $s=\frac12$ remained open for a complete characterization of  vanishing (non-vanishing resp.) geodesic distance for Sobolev metrics on the group of diffeomorphisms of a general manifold $M$.
This gap is closed by the following theorem.
The construction is illustrated in Fig.~\ref{fig:diffeo}.

\begin{theorem}\label{thm:vanishingDiff}
Assume the setting of Section~\ref{sec:sob}.
Then the right-invariant $H^s$ metric on $\Diff(M)$ has vanishing geodesic distance if and only if 
$s\leq \frac{\operatorname{dim}(M)}{2}$ and $s<1$, i.e., in dimension $1$ 
if and only if  $s \leq \frac12$, and in dimension 
$\geq 2$ if and only if $s<1$.
\end{theorem}

\begin{remark}
Different choices of Fourier multipliers, charts, and partitions of unity yield different but equivalent inner products and do not affect the degeneracy or non-degeneracy of the geodesic distance. 	
\end{remark}

\begin{figure}[h]
\centering
\begin{psfrags}
\psfrag{A}[cc][cc]{$A$}%
\psfrag{B}[cc][cc]{$A'$}%
\psfrag{S0}[tc][tc]{$0$}%
\psfrag{S1}[tc][tc]{$1$}%
\psfrag{S2}[tc][tc]{$2$}%
\psfrag{Sm1}[tc][tc]{$-1$}%
\psfrag{W0}[cr][cr]{$0$}%
\psfrag{W1}[cr][cr]{$1$}%
\psfrag{W2}[cr][cr]{$2$}%
\psfrag{Wm1}[cr][cr]{$-1$}%
\psfrag{x0}[tc][tc]{$0$}%
\psfrag{x2}[tc][tc]{$0.2$}%
\psfrag{x4}[tc][tc]{$0.4$}%
\psfrag{xm2}[tc][tc]{$-0.2$}%
\psfrag{xm4}[tc][tc]{$-0.4$}%
\psfrag{y0}[cr][cr]{$0$}%
\psfrag{y2}[cr][cr]{$0.2$}%
\psfrag{y4}[cr][cr]{$0.4$}%
\psfrag{ym2}[cr][cr]{$-0.2$}%
\psfrag{ym4}[cr][cr]{$-0.4$}%
\includegraphics[width=0.5\textwidth]{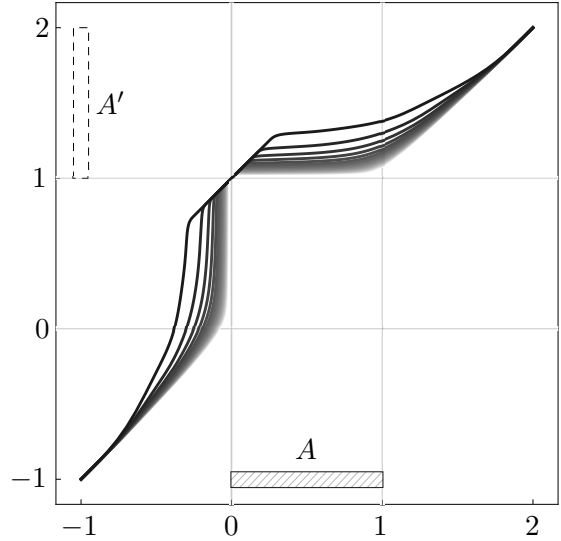}
\end{psfrags}
\caption{Illustration of the proof of Theorem~\ref{thm:vanishingDiff}: a sequence of diffeomorphisms with small $H^{1/2}$ distance to the identity; each diffeomorphism maps the set $A=(0,1)$ to some disjoint set $A'$ above the line $\{y=1\}$.}
\label{fig:diffeo}
\end{figure}

\begin{proof}[Proof of Theorem~\ref{thm:vanishingDiff}]
It suffices to show the theorem for $s=\frac12$ and $M=\mathbb R$. All other cases follow from \cite{BBHM2013,BBM2013,JM2018}. The proof is divided in three steps. For brevity, we write $\norm{h}_\varphi=\sqrt{\llangle h,h\rrangle_\varphi}$ for all $\varphi \in \Diff(M)$ and $h \in T_\varphi\Diff(M)$.

\noindent {\bf Step 1.} 
We claim that the manifold $M=\mathbb R$ contains a non-empty open set $A$ which has vanishing displacement energy with respect to the action of $\Diff(M)$.
We will prove this claim for the set $A=(0,1)$.
We start from the observation that the constant vector field $u=1$ has right translations $\operatorname{Fl}_t^{u}(x)=x+t$ as flow, and that the set $A$ does not intersect its right-translation $\operatorname{Fl}_t^{u}(A)$ at time $t=1$.
To make the energy of the time dependent vector field arbitrarily small we choose a family $(\xi_n)_{n\in\mathbb N}$ of bump functions with properties \ref{prop1}--\ref{prop3} of Lemma~\ref{lem:bumpfunction} and define for each $n\in\mathbb N$ the compactly supported time-dependent vector field
\begin{equation*}
u_n(t,x)=u(t,x). \xi_n\big(x-\operatorname{Fl}_t^{u}(0)\big)=\xi_n(x-t).
\end{equation*}
The idea behind this definition is to localize the vector field $u$ without affecting the trajectory of the point zero; see Fig.~\ref{fig:bump}.
Indeed, the trajectory of zero is given by $\operatorname{Fl}^{u}_t(0)=\operatorname{Fl}^{u_n}_t(0)$ because $\xi_n(0)=1$.
Note that the localization also corrects for the fact that right translations are not compactly supported.
Let $\varphi_n=\operatorname{Fl}^{u_n}_1 \in \Diff(\mathbb R)$.
As $\varphi_n$ preserves monotonicity, one has
\begin{equation*}
\forall x,y\in A:
\qquad
x<1=\varphi_n(0)<\varphi_n(y),
\end{equation*}
which proves that $\varphi_n(A)\cap A=\emptyset$.
Moreover, the $H^{1/2}$-distance between the identity and $\varphi_n$ tends to zero as $n$ tends to infinity:
\begin{align}
d(\id, \varphi_n)
\leq
\int_0^1 \| \xi_n(x-t)\|_{H^{1/2}(\mathbb R)} dt
=
\| \xi_n\|_{H^{1/2}(\mathbb R)}
\underset{n\to\infty}{\longrightarrow}0\;.
\end{align}
Thus, $A$ is an open set with vanishing displacement energy as claimed.

\noindent {\bf Step 2.}
Left multiplication $L_\varphi:\Diff(\mathbb R)\to\Diff(\mathbb R)$ is smooth for each $\varphi\in\Diff(\mathbb R)$.
Moreover, there is $C_\varphi>0$ such that for each vector field $X \in \mathfrak X(M)$,
\begin{align}
\label{equ:left_mult}
\| TL_{\varphi}X\|_{\varphi}
&=
\| TL_{\varphi}X\circ\varphi^{-1}\|_{\id}
=
\|d\varphi\circ\varphi^{-1}.X\circ\varphi^{-1}\|_{\id}
\leq
C_\varphi \|X\|_{\id}\;,
\end{align}
by the continuity of reparametrizations $H^{1/2}(\mathbb R) \ni X \mapsto R_{\varphi^{-1}}X\in H^{1/2}(\mathbb R)$ \cite[Lemma~B.3]{inci2013regularity} and the continuity of pointwise multiplications $H^{1/2}(\mathbb R)\ni X \mapsto d\varphi. X \in H^{1/2}(\mathbb R)$ \cite[Corollary in Section~4.2.2]{Tri1992}.
If $\psi\colon[0,1] \to \Diff(\mathbb R)$ is a smooth path and $X(t)=\partial_t\psi(t)\circ\psi(t)^{-1}$, this implies
\begin{align*}
\int_0^1 \| \partial_t(\varphi\circ\psi)\|_{\varphi\circ\psi}dt
&=
\int_0^1 \| TL_\varphi \partial_t\psi\|_{\varphi\circ\psi} dt
\\&=
\int_0^1 \| TL_\varphi X\|_{\varphi} dt
\leq
C_\varphi \int_0^1 \|X\|_{H^s} dt
=
C_\varphi \int_0^1 \|\partial_t\psi\|_\psi dt\;.
\end{align*}
Taking the infimum over all paths $\psi$ with fixed end points shows \eqref{ass:lipschitz}.
Finally we note that for each non-empty set the group $\Diff(A)$ is non-Abelian, c.f. \cite{ElPo1993}. 
 Thus, the conditions of Theorem \ref{thm:displacement} are satisfied for the group $\Diff(\mathbb R)$ with the right-invariant $H^s$ metric $\llangle\cdot,\cdot\rrangle$.
Moreover, $\Diff(\mathbb R)$ is simple by \cite[Theorem~2.1.1]{Ban1997}.
Thus, Theorem~\ref{thm:displacement} together with the result of Step~2 show that the geodesic distance vanishes identically on $\Diff(\mathbb R)$.
\end{proof}

\section{Groups of volume preserving diffeomorphisms and the SQG equation}\label{sec:volumepreserving}

Recall that the SQG equation is the Euler-Arnold equation of the right invariant $H^{-1/2}$ metric on the group of diffeomorphisms which preserve the volume (or equivalently symplectic) form of a two-manifold \cite{Was2016}. We prove in this section that the geodesic distance associated to this metric vanishes. More generally, we show that this result extends to groups of exact diffeomorphisms on higher-dimensional manifolds.

\subsection{Sobolev metrics on groups of exact diffeomorphisms}\label{sec:exact}
Let $M$ be a be a connected finite-dimensional manifold endowed with a Riemannian metric $\langle\cdot,\cdot\rangle$ of bounded geometry and a symplectic form $\omega$.
The symplectic gradient of a function $f \in C^\infty(M)$ is denoted by $\nabla^\omega f=\check\omega^{-1}df \in \mathfrak X(M)$, where $\check\omega\colon TM\to T^*M$ is the symplectic isomorphism.
A vector field is called exact if it is the symplectic gradient of a compactly supported function, and a diffeomorphism is called exact if it is generated by a time-dependent symplectic vector field, i.e.,
\begin{align*}
\mathfrak X_\ex(M)&=\{\nabla^\omega f: f\in C^\infty_c(M)\},
\\
\Diff_\ex(M)&=\{\varphi(1):\varphi \in C^\infty([0,1],\Diff(M)), \forall t \in [0,1]:\varphi'(t)\circ\varphi(t)^{-1}\in\mathfrak X_\ex(M)\}.
\end{align*}
Alternative common names are globally Hamiltonian vector fields and Hamiltonian diffeomorphisms.
Assume that $\Diff_\ex(M)$ is a convenient Lie group with Lie algebra $\mathfrak X_\ex(M)$.
This assumption is satisfied if $M$ is compact \cite{ratiu1981differentiable} or, more generally, if $M$ is connected and separable and the vector space of exact compactly supported $1$-forms is a direct summand in the space of all closed compactly supported $1$-forms \cite[Sect.~43.13]{KM1997}.
Then the $H^s$ metric on $\Diff_\ex(M)$ is defined as the unique right-invariant Riemannian metric $\llangle\cdot,\cdot\rrangle$ which satisfies
\begin{align*}
\llangle \nabla^\omega f\circ\varphi,\nabla^\omega f\circ\varphi\rrangle_\varphi
=
\llangle \nabla^\omega f,\nabla^\omega f\rrangle_{\id}
=
\|f\|_{\dot H^{s+1}(M)}^2,
\qquad
\forall \varphi\in\Diff_\ex(M),\forall f \in C^\infty_c(M).
\end{align*}
Here $\|\cdot\|_{\dot H^{s+1}(M)}$ denotes the homogeneous Sobolev (pseudo) norm of order $s+1$.
For $s=-1$ this yields the bi-invariant metric as studied by Eliashberg and Polterovich \cite{ElPo1993}.
The corresponding geodesic distance is defined as in Section~\ref{sec:geo}.
Note that $\Diff_\ex(M)$ is a subgroup of the group $\Diff_\omega(M):=\{\varphi \in \Diff(M): \varphi^*\omega=\omega\}$ of symplectic diffeomorphisms.

\subsection{Vanishing geodesic distance on exact diffeomorphisms}

The geodesic distance of the $H^s$ metric on $\Diff_\ex(M)$ is known to be non-degenerate for $s\geq 0$ by \cite{MM2005} and degenerate for $s=-1$ by \cite{ElPo1993}. The following theorem shows degeneracy for $s\leq -\frac12$.

\begin{theorem}\label{thm:maintheorem}
Assume the setting of Section~\ref{sec:exact} with $s\leq-1/2$. Then the geodesic distance of the right-invariant $H^{s}$ metric vanishes identically on the commutator sub-group $[\Diff_\ex(M),\Diff_\ex(M)]$ and, if $M$ is compact, on $\Diff_\ex(M)$.
\end{theorem}

\begin{proof}
\noindent {\bf Step 1.} We claim that the manifold $M=\mathbb R^2$ with the canonical symplectic form $\omega=dx\wedge dy$ contains a non-empty open subset $A$ with vanishing displacement energy.
To prove this claim we consider a bump function $\psi\in C^\infty(\mathbb R,[0,1])$ satisfying
for each $x \in \mathbb R$ that
\begin{equation}\label{equ:psi}
\psi(x)=\begin{cases} 1,\qquad x\in (-1,1)\\
0,\qquad x\notin (-2,2)\;.
\end{cases}
\end{equation}
and define the Hamiltonian function
\begin{equation*}f\colon\mathbb R^2\to\mathbb R,
\qquad
f(x,y)= -x\psi(x)\;.
\end{equation*}
Then the symplectic gradient and gradient flow of $f$ are given by
\begin{align*}
u(x,y)
&:=
\nabla^{\omega}f(x,y)
=
(\partial_y f(x,y),-\partial_x f(x,y))
=
(0,\psi(x)+x \psi'(x))\;,	
\\
\operatorname{Fl}_t^u(x,y)&=(x, y+t (\psi(x)+x\psi'(x))\;.
\end{align*}
Note that $\operatorname{Fl}_1^u$ is an exact diffeomorphism which maps the set $A=(-1,1)\times(0,1)$ to the disjoint set $\operatorname{Fl}_1^u(A)=(-1,1)\times(1,2)$; c.f.\@ Fig.~\ref{fig:volume_preserving}.

We will now shorten the $H^{1/2}$ length of the flow of $u$ by modifying $u$ suitably.
Let $(\xi_n)_{n\in\mathbb N}$ be a sequence of smooth bump functions with properties \ref{prop1}--\ref{prop3} of Lemma~\ref{lem:bumpfunction},
and let $g(t,x)$ describe the vertical position of the point $(x,0)$ under the flow of $u$ at time $t$, i.e.,
\begin{equation*}
g(t,x)=\operatorname{pr}_2 \operatorname{Fl}^u_t(x,0)=t \big(\psi(x)+x\psi'(x)\big)\;.
\end{equation*}
Then we define for each $n\in\mathbb N$ a time-dependent Hamiltonian function $f_n$ and vector field $u_n$ by
\begin{equation}
f_n(t,x,y)= f(x) \xi_n\big(y-g(t,x)\big)\;,
\qquad
u_n(t,x,y) = \nabla^\omega f_n(t,x,y)\;,
\end{equation}
where $\nabla^\omega=(\partial_y,-\partial_x)$ acts only in the spatial dimensions. As $u_n$ coincides with $u$ along $y=g(t,x)$, the corresponding flow satisfies
\begin{align*}
\operatorname{Fl}^{u_n}_t(x,0)=\operatorname{Fl}^u_t(x,0)=(x,g(t,x))\;.
\end{align*}
Thus, as illustrated in Fig.~\ref{fig:volume_preserving}, the line $\{(x,g(1,x)): x\in\mathbb R\}$ lies above the set $A$ and below the set $\operatorname{Fl}^{u_n}_1(A)$.
It follows that $\operatorname{Fl}^{u_n}_1(A)\cap A=\emptyset$.
The $H^{-1/2}$ length of the flow of $u_n$ can be estimated using $f_n(t,x,y)=(f\otimes\xi_n)\circ\operatorname{Fl}^{-u}_t(x,y)$ as
\begin{align*}
\int_0^1 \|u_n\|_{H^{-1/2}(\mathbb R^2)}dt
&=
\int_0^1 \|f_n\|_{H^{1/2}(\mathbb R^2)}dt
\leq
\int_0^1 \|f\otimes\xi_n\|_{H^{1/2}(\mathbb R^2)}\|R_{\operatorname{Fl}^{-u}_t}\|_{L(H^{1/2}(\mathbb R^2))} dt
\\&=
\|f\|_{H^{1/2}(\mathbb R)}\|\xi_n\|_{H^{1/2}(\mathbb R)} \int_0^1 \|R_{\operatorname{Fl}^{-u}_t}\|_{L(H^{1/2}(\mathbb R^2))} dt,
\end{align*}
where the inequality follows from the continuity of compositions by diffeomorphisms \cite[Lemma~2.7]{inci2013regularity} and the last equality from the Hilbert tensor product representation $H^{1/2}(\mathbb R^2)=H^{1/2}(\mathbb R)\hat\otimes H^{1/2}(\mathbb R)$ \cite[Theorem~2.1]{sickel2009tensor}.
Thus, the $H^{-1/2}$ distance $d(\id,\operatorname{Fl}^{u_n}_1)$ tends to zero as $n\to\infty$, which shows that the displacement energy of $A$ vanishes.

\noindent {\bf Step 2.} We claim that every symplectic manifold $M$ contains a non-empty open subset with vanishing displacement energy.
To prove the claim, note that any Darboux coordinate system defines a symplectomorphism between an open subset $U$ of $M$ and an open subset $V$ of $\mathbb R^{2d}$, where $\mathbb R^{2d}$ carries the canonical symplectic form $\sum_{i=1}^d dx_{2i-1}\wedge dx_{2i}$.
Without loss of generality $V$ is a box $(-2r,2r)^{2d}$ for some $r>0$.
Let the bump function $\psi$ and the Hamiltonian functions $f_n$ be defined as in Step~1,
choose $\epsilon>0$ such that the Hamiltonian functions $(x_1,x_2) \mapsto f_n(t,x_1/\epsilon,x_2/\epsilon)$ are supported in $(-r,r)^2$,
and define the localized Hamiltonian function
\begin{equation*}
g_n(t,x) = f_n(t,x_1/\epsilon,x_2/\epsilon) \prod_{i=3}^{2d} \psi(x_i/r)\;,
\qquad
t \in [0,1]\;, \qquad x=(x_1,\dots,x_{2d}) \in \mathbb R^{2d}\;.
\end{equation*}
Note that $g_n$ is supported in $V$ and equals $f_n(t,x_1/\epsilon,x_2/\epsilon)$ on $V/2$.
If one sets
\begin{equation*}
B = (-\epsilon,\epsilon)\times(0,\epsilon) \times (-r,r)^{2d-2}\;,
\end{equation*}
then it follows from Step~1 that the flow of the symplectic gradient $v_n(t,x)=\nabla^\omega g_n(t,x)$ satisfies $\operatorname{Fl}^{v_n}_1(B) \cap B=\emptyset$.
Moreover, the $H^{-1/2}$ length of the flow of $v_n$ can be estimated as follows: by the Hilbert tensor product representation $H^{1/2}(\mathbb R^{2d})=H^{1/2}(\mathbb R)\hat\otimes\cdots\hat\otimes H^{1/2}(\mathbb R)$ \cite[Theorem~2.1]{sickel2009tensor},
\begin{align*}
\int_0^1 \|v_n\|_{H^{-1/2}(\mathbb R^{2d})} dt
&=
\int_0^1 \|g_n\|_{H^{1/2}(\mathbb R^{2d})} dt
\\&=
\|x\mapsto\psi(x/r)\|_{H^{1/2}(\mathbb R)}^{2d-2} \int_0^1 \|(x_1,x_2)\mapsto f_n(t,x_1/\epsilon,x_2/\epsilon)\|_{H^{1/2}(\mathbb R^2)} dt,
\end{align*}
where the right-hand side tends to zero as $n \to \infty$ by Step~1.
Transferring this result from $V\subseteq\mathbb R^{2d}$ back to $U\subseteq M$ using the Darboux coordinates proves the claim.

\noindent {\bf Step 3.}
The assumptions of Theorem \ref{thm:displacement} are satisfied for the group $\Diff_\ex(M)$ with the right-invariant $H^s$ metric $\llangle\cdot,\cdot\rrangle$.
This can be verified as in the proof of Theorem~\ref{thm:vanishingDiff}.
As any non-empty open set is essential for the action of $\Diff_\ex(M)$ \cite[Lemma~2.1.12 and Theorem~2.3.1]{Ban1997}, the implication \ref{item:displacement3} $\Rightarrow$ \ref{item:displacement1} of Theorem~\ref{thm:displacement} shows the existence of a non-trivial $\varphi \in \Diff_\ex(M)$ with vanishing distance to the identity.
The proof of Theorem~\ref{thm:displacement} actually reveals the stronger statement that $\varphi$ belongs to the commutator subgroup $[\Diff_\ex(M),\Diff_\ex(M)]$.
For compact $M$ the group $\Diff_\ex(M)$ is simple \cite[Theorem 4.3.1 and Remark~4.2.3]{Ban1997}, and for non-compact $M$ the commutator subgroup $[\Diff_\ex(M),\Diff_\ex(M)]$ is simple \cite[Theorem 4.3.3]{Ban1997}.
Thus, the geodesic distance vanishes identically on these respective groups by Theorem~\ref{thm:displacement}.\ref{item:displacement4}.
\end{proof}

\begin{figure}
\centering
\begin{psfrags}
\psfrag{S0}[tc][tc]{$0$}%
\psfrag{S1}[tc][tc]{$1$}%
\psfrag{S2}[tc][tc]{$2$}%
\psfrag{Sm1}[tc][tc]{$-1$}%
\psfrag{Sm2}[tc][tc]{$-2$}%
\psfrag{W0}[cr][cr]{$0$}%
\psfrag{W1}[cr][cr]{$1$}%
\psfrag{W2}[cr][cr]{$2$}%
\psfrag{Wm1}[cr][cr]{$-1$}%
\psfrag{Wm2}[cr][cr]{$-2$}%
\psfrag{Wm3}[cr][cr]{$-3$}%
\psfrag{x0}[tc][tc]{$0$}%
\psfrag{x11}[tc][tc]{$1$}%
\psfrag{x2}[tc][tc]{$0.2$}%
\psfrag{x4}[tc][tc]{$0.4$}%
\psfrag{x6}[tc][tc]{$0.6$}%
\psfrag{x8}[tc][tc]{$0.8$}%
\psfrag{y0}[cr][cr]{$0$}%
\psfrag{y11}[cr][cr]{$1$}%
\psfrag{y2}[cr][cr]{$0.2$}%
\psfrag{y4}[cr][cr]{$0.4$}%
\psfrag{y6}[cr][cr]{$0.6$}%
\psfrag{y8}[cr][cr]{$0.8$}%
\setlength{\unitlength}{0.5\textwidth}
\begin{picture}(1,1)
\put(0,0){\includegraphics[width=0.5\textwidth]{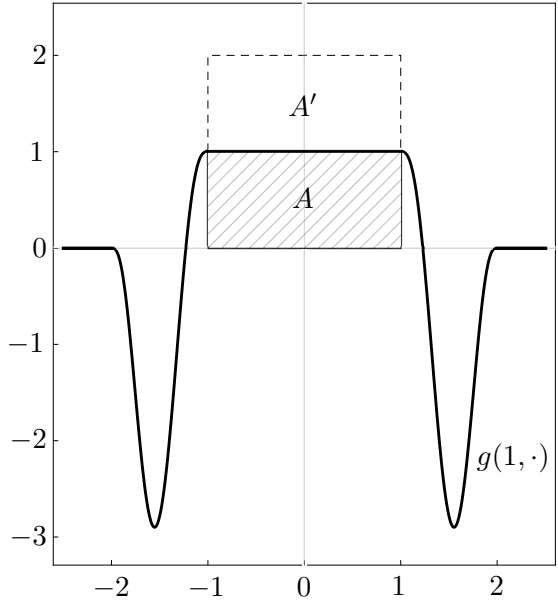}}
\put(0.5, 0.82){\makebox(0,0)[c]{$A'$}}
\put(0.5, 0.67){\makebox(0,0)[c]{$A$}}
\put(0.78, 0.25){\makebox(0,0)[l]{$g(1,\cdot)$}}
\end{picture}
\end{psfrags}
\caption{Illustration of the proof of Theorem~\ref{thm:vanishingDiff}. The exact diffeomorphism $\operatorname{Fl}^u_1$ maps the set $A$ to $A'$ and the line $\{y=0\}$ to $\{y=g(1,x)\}$. The exact diffeomorphisms $\operatorname{Fl}^{u_n}_1$ also map the set $A$ to some set above the line $\{y=g(1,x)\}$ and additionally have short $H^{1/2}$ distance to the identity.}
\label{fig:volume_preserving}
\end{figure}

\subsection{Degenerate geodesic distance on volume preserving diffeomorphisms}
On two dimensional manifolds, volume forms coincide with symplectic forms.
This allows one to apply Theorem~\ref{thm:vanishingDiff} to groups of volume preserving diffeomorphisms, which are of particular interest because several prominent PDEs are Euler--Arnold (i.e., geodesic) equations of Sobolev $H^s$ metrics thereon: for $s=0$ one obtains Euler's equation for the motion of an ideal fluid \cite{Arn1966}, and for $s=-\frac{1}{2}$ one obtains the SQG-equation \cite{Was2016}.
The following corollary to Theorem~\ref{thm:vanishingDiff} states that the SQG equation corresponds to a degenerate Riemannian metric.
Note that this is in stark contrast to Euler's equation, which corresponds to a non-degenerate metric.

\begin{corollary}\label{cor:volume}
Let $(M,g)$ be a two-dimensional orientable Riemannian manifold, and let $\mu=\vol(g)$ be the Riemannian volume form.
Then the geodesic distance of the right-invariant $H^{-1/2}$ metric on the group $\Diff_\mu(M):=\{\varphi\in\Diff(M):\varphi^*\mu=\mu\}$ of volume preserving diffeomorphisms is degenerate.
\end{corollary}

\begin{proof}
As $M$ is two-dimensional, the volume form $\mu$ is also a symplectic form.
Step~2 in the proof of Theorem~\ref{thm:maintheorem} shows that $M$ contains an open set $A$ with vanishing displacement energy with respect to the action of $\Diff_\ex(M)$.
Note that the assumption that $\Diff_\ex(M)$ is a convenient Lie group is not needed here.
The set $A$ has vanishing displacement energy also with respect to the action of $\Diff_\mu(M)$ because $\Diff_\ex(M)$ is contained in $\Diff_\mu(M)$ and because $\|\nabla^\omega f\|_{H^{-1/2}(\mathbb R^2)}=\|f\|_{H^{1/2}}$ for each $f \in C^\infty_c(M)$.
Thus, the geodesic distance on $\Diff_\mu(M)$ is degenerate by the implication \ref{item:displacement3} $\Rightarrow$ \ref{item:displacement1} of Theorem~\ref{thm:displacement} applied to $G=\Diff_\mu(M)$.
\end{proof}

\section{Open problems and conjectures}\label{sec:conjecture2}
\subsection{Degeneracy of the geodesic distance on diffeomorphism groups}
The present article and the recent article \cite{JM2018} by Jerrard and Maor give a complete characterization of the geodesic distance of right invariant $H^s$ metrics on diffeomorphism groups (see Theorem~\ref{thm:vanishingDiff}). 
Jerrard and Maor consider not only metrics of type $H^s=W^{s,2}$, but also of type $W^{s,p}$ for general $p \in [1,\infty)$. 
In \cite{JM2018} they obtained a nearly complete characterization for this class of metrics; only the behavior at the critical Sobolev index $s=\frac{\dim{M}}{p}$ remained open. 
One of the main difficulties in their analysis was to control the end points of certain paths in $\Diff(M)$ with arbitrarily short length. 
Such precise control of the end points is not needed if Theorem~\ref{thm:displacement} is invoked. 
This simplification was used in their follow-up article \cite{JM2019} to complete the characterization of vanishing (non-vanishing, resp.) geodesic distance.

For symplectomorphism groups much less is known. The results of this article show that the geodesic distance of $H^s$ metrics is degenerate for $s\leq-\frac12$ by Theorem~\ref{thm:maintheorem} and non-degenerate for $s\geq 0$ by \cite[Theorem~5.7]{MM2005}, but the case $-\frac12< s< 0$ remains open. 
Sobolev metrics of type $W^{s,p}$ have been studied only for $s=-1$, and there the geodesic distance vanishes if and only if $p<\infty$. The case $s=-1$ and $p=\infty$ corresponds to the Hofer metric, which is known to have non-vanishing geodesic distance \cite{hofer1990topological}.
An interesting question is if the critical index is independent of the dimension of $M$. 
This is certainly compatible with existing results, and might be due to the higher rigidity of the group of symplectomorphisms. This would suggest that the geodesic distance of the right invariant $W^{s,p}$ metric with  $p\in [1,\infty)$ is degenerate if and only if $s\leq \frac1p-1$.

\subsection{Relation to Fredholmness of the exponential map}

By the monotonicity of the $H^s$ distance in $s$, there is a critical threshold $s^*_{\text{dist}}$ such that the $H^s$ geodesic distance is degenerate below and non-degenerate above the threshold: 
\begin{equation*}
s^*_{\text{dist}} = \sup \{s\in \mathbb R: d_s(g_1,g_2) = 0, \forall g_1,g_2\}\;.
\end{equation*}
Similarly, smoothness of geodesic spray depends monotonically on $s$ \cite{BEK2014,EK2014}. Thus, there is a again a critical index 
\begin{equation*}
s^*_{\text{smooth}} = \inf \{s\in \mathbb R: \operatorname{exp} \text{ is } C^1\}\;,
\end{equation*}
where $\exp$ denotes the Riemannian exponential map of the $H^s$-metric.
Moreover, the Fredholm property of the exponential map is monotonic in $s$, since Fredholmness is generally obtained by compactness of the operator $v\mapsto \ad_v^{\star}u_0$ describing the Euler-Arnold equation on diffeomorphism groups. Compactness of any such operator implies compactness of all operators of higher order \cite{MP2010}. Thus, there is again a critical threshold $s^*_{\text{Fred}}$ for Fredholmness, 
$$ s^*_{\text{Fred}} = \inf \{s\in \mathbb{R}: v\mapsto \ad_v^{\star}u_0 \text{ is compact}\}\;.$$
We conjecture that these  thresholds are connected to each other as follows:
\begin{conjecture}\label{con:fredholm}
On any group of diffeomorphisms the critical thresholds satisfy the relation
\begin{equation}
s^*_{\operatorname{smooth}}=s^*_{\operatorname{Fred}}\leq s^*_{\operatorname{dist}}.
\end{equation}
\end{conjecture}
This conjecture holds true in all known cases, including groups of diffeomorphisms, volume preserving diffeomorphisms and symplectomorphisms. It is, however, important to note that the behavior \emph{at} the critical index may vary from property to property. For example, for 1D diffeomorphisms at $s=\frac12$, Fredholmness is false, whereas smoothness of the exponential map is true \cite{BKP2016}. Similar statements apply to 3D fluids at $s=0$ \cite{MP2010}.

If Conjecture~\ref{con:fredholm} is true, then the Fredhomness results of \cite{MP2010} would imply the following complete characterization for volume preserving diffeomorphisms on manifolds of dimension $\geq 3$:
\begin{conjecture}
The geodesic distance of the right invariant $H^s$ metric on $\Diff_\mu(M)$ in dimension $\geq 3$ vanishes if and only if $s<0$.
\end{conjecture}

\subsection{Relations to longtime existence of solutions to the geodesic equation}

Properties of the Riemannian distance seem associated to global existence phenomena of the corresponding geodesic and Euler--Arnold equations. In finite dimensions this is expected: the Hopf--Rinow theorem says that geodesics extend for all time if and only if the Riemannian distance function gives a complete metric space. In infinite dimensions this is much less well-understood, and so far there exist no formal results in this direction. We see however several aspects of ``borderline'' behavior depending on the smoothness parameter $s$: the transition between global conservative (weak) solutions and nonunique shock solutions; between having a nonsmooth exponential map and having a smooth one; and between vanishing geodesic distance and positive geodesic distance. Many of these transitions seem to happen at the same value of $s$, based on the limited information we still have about the complete picture.

For example, in one space dimension, the Euler--Arnold equations include the Hunter--Saxton and Camassa--Holm equations at $s=1$, the Wunsch or modified Constantin--Lax--Majda equation at $s=\frac12$, and the inviscid Burgers' equation at $s=0$ (see Table~\ref{tab:1} for an overview and references). All of these have solutions that blow up in finite time. In the case $s=1$, solutions of the Hunter--Saxton~\cite{Len2007a,BBM2014a} and Camassa--Holm~\cite{lee2018global, McK2004} equations may be continued using a geometric transformation in the space of smooth maps. Moreover, the exponential map is smooth, and the Riemannian distance is positive. On the other hand, in the case $s=0$, solutions of the inviscid Burgers' equation exhibit genuine shocks and their flows must lose continuity (and in particular smoothness) as well as dissipating energy; in addition the exponential map is non-smooth and the Riemannian distance vanishes. In between lies the Wunsch equation at $s=1/2$. Here the exponential map is smooth, the Riemannian distance vanishes, and all geodesics end in finite time. It is not known whether geodesics can be continued in a slightly larger space of smooth or continuous nondecreasing functions, or whether they must leave the space of continuous functions entirely. 

In future work we aim for a better understanding of the blowup properties of the borderline case and connecting them to the results of this article.

\subsection*{Acknowledgments}
We would like to thank Martins Bruveris, Stefan Haller, Robert Jerrard, Cy Maor, Peter Michor, and Gerard Misio{\l}ek  for helpful comments and discussions during the preparation of this manuscript.

%\bibliographystyle{abbrv}
%\bibliography{sqg}

\def\cprime{$'$}

\end{document}